\documentclass[draft,11pt]{article}
\usepackage{amssymb,amsmath,amsfonts,amsthm}
\theoremstyle{plain}
\newtheorem{theorem}{Theorem}
\newtheorem{corollary}{Corollary}
\newtheorem{lemma}{Lemma}
\newtheorem{proposition}{Proposition}

\theoremstyle{definition}
\newtheorem{definition}[theorem]{Definition}

\theoremstyle{remark}

\def\T{{\mathcal T}}
\def\TILT{\tilde{\mathcal T}}
\def\tilt{\tilde{t}}

\makeatother 
\def\W{{\bf W}}

\def\cT {{\cal T}}
\input epsf

\def\wram{{\rm wR}}

\title{On weighted Ramsey numbers}
\author{Maria Axenovich and Ryan Martin\thanks{Corresponding
author. The second author's research supported in part by
NSA grant H98230-05-1-0257.} \\
\small Department of Mathematics\\[-0.8ex]
\small Iowa State University\\[-0.8ex]
\small  USA\\[-0.8ex]
\small \texttt{axenovic@iastate.edu, rymartin@iastate.edu} }

\begin{document}
\maketitle

\begin{abstract}
The weighted Ramsey number, $\wram(n,k)$, is the minimum  $q$ such
that there is an  assignment of nonnegative real numbers (weights)
to the edges of $K_n$ with the total sum of the weights equal to
$\binom{n}{2}$ and there is a Red/Blue coloring of edges of the
same $K_n$, such that in any complete $k$-vertex subgraph $H$, of
$K_n$, the sum of the weights on Red edges in $H$ is at most $q$
and the sum of the weights on Blue edges in $H$  is at most $q$.
This concept was introduced recently by Fujisawa and Ota.

We provide new bounds on $\wram(n,k)$, for $k\geq 4$ and $n$ large
enough and show that determining $\wram(n,3)$ is asymptotically
equivalent to the problem of finding the fractional packing number
of monochromatic triangles in colorings of edges  of complete graphs with two colors.
\end{abstract}

\parindent 0pt
\parskip 8pt
\section{Introduction}


\begin{definition}
The weighted Ramsey number, $\wram(n,k)$, is the minimum  $q$ such
that there is an  assignment of nonnegative real numbers (weights)
to the edges of $K_n$ with the total sum of the weights equal to
$\binom{n}{2}$ and there is a Red/Blue coloring of the edges of the
same $K_n$, such that in any complete $k$-vertex subgraph $H$ of
$K_n$, the sum of the weights on Red edges in $H$ is at most $q$ and
the sum of the weights on Blue edges in $H$  is at most $q$.
\end{definition}

This notion was introduced by Fujisawa and Ota in~\cite{FO}, where
the authors used the scaled version of the above definition
requiring the total sum of weights to be $1$;  the
corresponding weighted Ramsey function from \cite{FO} is
$\wram(n,k)/ \binom{n}{2}$. The main results obtained in
\cite{FO} can be summarized as follows.

\begin{theorem}[\cite{FO}] \label{general-FO}
For any integers $n,k$, $4\leq k\leq n$,
$$ \frac{1}{2}\binom{k}{2} \leq \wram(n,k)
   <\frac{k^2-1}{k^2+1}\binom{k}{2}. $$
In addition, $\wram(5,3)=2$ and $\wram(n,3)\geq 15/7$ for $n\geq
6$ (with equality when $n=6$). Moreover, the better asymptotic bound holds:  $\wram(n,3)\geq 110/49-o(1)$.
\end{theorem}

The main emphasis of \cite{FO} was determining the bound on
$\wram(n,3)$, where the authors observed a relation between
weighted Ramsey numbers and the edge-disjoint packing of
monochromatic triangles in a $2$-colored complete graph. For an edge coloring $c$ of a complete graph with two colors, let $\tau(c,3)$ be the
largest number of edge-disjoint monochromatic triangles in $c$.
Let $$\tau(n,3) = \min \{ \tau(c,3): ~ c \mbox{ is a }
2-\mbox{edge-coloring of } K_n\}.$$ The following was proven in
\cite{FO}:

\begin{theorem}[\cite{FO}]\label{triangle-FO}
$$\wram(n, 3)\geq \frac{4\binom{n}{2}}{n^2- 2\tau(n,3) +n}.$$
\end{theorem}

Together with the bound $\tau(n,3) \geq (\frac{3}{55} +o(1))n^2$
given in~\cite{EFGJL} and the upper bound of Theorem
\ref{general-FO}, the authors of~\cite{FO} provide the following bound:
\begin{equation}\label{newlab}
2.2448+o(1) \leq \wram(n,3)<2.4 .
\end{equation}
Note that more recent  better bound $\tau(n,3)\geq
(\frac{1}{12.888}+o(1))n^2$ given by Keevash and Sudakov in
\cite{KS}, immediately improves (\ref{newlab}) as follows.

\begin{equation}
2.3674+o(1) \leq \wram(n,3)<2.4 .
\end{equation}

Moreover, as also observed in \cite{FO}, if the value conjectured by
Erd\H{o}s, $\tau(n,3)=(\frac{1}{12}+o(1))n^2$, is correct, then
$\wram(n,3)$ would be  asymptotically  equal to $2.4$.

In Theorem~\ref{general-our}, we treat the general case
$\wram(n,k)$ for $k\geq 4$. We obtain better bounds and relate the
weighted Ramsey problem to Tur\'an-Ramsey type results using the regularity lemma
of Szemer\'edi. In Theorem~\ref{triangle-our}, we analyze
$\wram(n,3)$ and related problems. We show that finding
$\wram(n,3)$ is asymptotically equivalent to finding the
fractional packing number of monochromatic triangles in
$2$-colored complete graphs.

We choose to make the total sum of the edge-weights
equal to $\binom{n}{2}$, instead of $1$ as in~\cite{FO}, in order for easier analysis of the asymptotic behavior of
$\wram(n,k)$.  In fact, to state our main results, we use the {\em
weighted Ramsey limit} defined as follows:
$$\W(k):= \lim_{n\rightarrow \infty} \wram(n,k).$$
We prove the existence of this limit in
Section~\ref{linear-programs}. Note that Theorem~\ref{general-FO}
gives that $\frac{k-1}{k}\left\lfloor k^2/4\right\rfloor\leq
\W(k)\leq 2\left\lfloor k^2/4\right\rfloor$.  Our main theorem is
the following.

\begin{theorem}\label{general-our}
Let $k$ be an integer, $k\geq 5$. \begin{equation}
1.051\left\lfloor\frac{k^2}{4}\right\rfloor< \W(k)
   \leq 1.25\left\lfloor\frac{k^2}{4}\right\rfloor .
   \label{gen-eq1} \end{equation}
If $k$ is sufficiently large, \begin{equation} \W(k)>
1.059\left\lfloor\frac{k^2}{4}\right\rfloor . \label{gen-eq2}
\end{equation}

Moreover, more accurate bounds for small $k$ can be summarized in
the following table, where $U(k)$ and $L(k)$ are the upper and
lower bounds on $\W(k)$, respectively.\\
\begin{center}
\begin{tabular} {|r||r|r|r|r|r|} \hline
   $k$    & $4$       & $5$      & $6$      & $7$
          & $8$ \\ \hline
   $L(k)$ & $4.1999$  & $6.3572$ & $9.5197$ & $12.7091$
          & $16.9115$ \\ \hline
   $U(k)$ & $4.8$     & $7.5$    & $11.25$  & $15$
          & $20$ \\ \hline
\end{tabular}
\end{center}
\end{theorem}

Note that both (\ref{gen-eq1}) and (\ref{gen-eq2}) improve the
constants in the lower and upper bounds from previously known
constants close to $1$ and $2$, respectively. We conjecture that
the upper bound of $1.25\lfloor k^2/4\rfloor$ gives the correct
value for $k\geq 5$ but $1.2\lfloor k^2/4\rfloor$ is correct for
$k=3,4$.

For a graph $G$ on $n$ vertices  let $\T_3(G)$ denote the set of
triangles in $G$. Let each triangle be assigned a real number,
called a {\em weight},  between $0$ and $1$. We say that this
assignment is {\it proper} if, for each edge $e$ of $G$,  the sum of
weights of triangles containing $e$ is at most $1$. The {\em
fractional triangle packing number} of $G$, denoted   $\tau^*$, is
the largest possible total weight of edges in a proper weight
assignment. Formally, it is defined as follows.
\begin{eqnarray}
   \tau^*(G) & = & \max \sum_{T\in\T_3(G)}g(T)
   \label{LP:taustar} \\
   & & \mbox{such that } \left\{\begin{array}{rl}
                            \sum\limits_{\scriptsize
                                         \begin{array}{l} T\ni e \\
                                         T\in\T_3(G) \end{array}}
                            g(T)\leq 1, & \forall e\in E(G) ; \\
                            g(T)\geq 0, & \forall T\in
                            \T_3(G) .
                         \end{array}\right. \nonumber
\end{eqnarray}

Let $$\tau^*(n, 3) := \min \{ \tau^*(R)+\tau^*(B):  R \mbox{ and }
B \mbox{ are color classes in a $2$-edge-coloring of } K_n\}.$$
Let
$$ \tau^*(3) := \lim_{n \rightarrow \infty}
\frac{\tau^*(n,3)}{\binom{n}{2}}, \quad  \tau(3) := \lim_{n
\rightarrow \infty} \frac{\tau(n,3)}{\binom{n}{2}}. $$ The fact that
these limits are well-defined follows from the monotonicity and
boundedness of the corresponding functions, see, for example,
\cite{KS}.

\begin{theorem}\label{triangle-our}
$$\W(3) = \frac {2}{1-\tau^*(3)}.$$
\end{theorem}

Using the result of Haxell and R\"odl, see \cite{HR}, implying
that $\tau^*(n,3) = \tau(n,3)(1+o(1)),$ we have the following.

\begin{corollary}\label{triangle-coro}
$$\W(3) = \frac {2}{1-\tau(3)}.$$
\end{corollary}

In Section \ref{linear-programs} we define related linear programs
and prove the correspondence between those and the original problem
of finding $\wram(n,k)$, we also prove the existence of the weighted
Ramsey limit in that section.  We prove Theorem \ref{general-our} in
Section \ref{general}. In Section~\ref{triangles}, we treat the case
$k=3$ and prove Theorem \ref{triangle-our}.  For common graph theory
notation, see, for example,~\cite{West}.

\section{Defining the linear programs} \label{linear-programs}
We formulate several problems in terms of linear programs.  See
\cite{S} for the terminology. We say that a  $k$-vertex subgraph of
an edge-colored $K_n$  is a {\it mono-$k$-subgraph} if all its edges
have the same color.   Let $\cT(c; n, k)$ be a set of
mono-$k$-subgraphs of a coloring $c$ of $K_n$. Next, we define $r(c;
n,k)$, which calculates the maximal total sum of nonnegative real
values assigned to the edges of a $2$-colored $K_n$ such that the
sum of these values on the edges of each mono-$k$-subgraph is at
most $1$. Formally, it is defined  as follows.
\begin{eqnarray}
   r(c; n,k) & = & \max\;\;\sum\limits_{e\in E(K_n)}w(e) \label{r} \\
   & & \begin{array}{rl}
          \mbox{ such that } & \left\{\begin{array}{rl}
                                 \sum\limits_{e\in E(T)}w(e) \leq 1, &
                                 \forall T\in \cT(c;n,k) ; \\
                                 w(e) \geq 0, &
                                 \forall e\in E(K_n)
                              \end{array}\right.
       \end{array} \nonumber \\
\end{eqnarray}

Let
$$r(n,k)= \max \{r(c; n,k):   c \mbox{ is a Red/Blue coloring of } K_n\}.$$

The following lemma will allow us to study $\wram(n,k)$ using the
more convenient function $r(n,k)$.

\begin{lemma} For any integers, $k$ and $n$, $3\leq k\leq n$,
$$\wram(n,k)= \frac{\binom{n}{2}}{r(n,k)}.$$
\label{lem:equivalence}
\end{lemma}

\begin{proof}
To show an upper bound, assume that
$\wram(n,k)>\binom{n}{2}/r(n,k)$.  Thus, for any Red/Blue coloring
$c$ of $K_n$ and any weight assignment to its edges with total sum
$\binom{n}{2}$, we have that there is a mono-$k$-subgraph with total
weight on its edges at least $q$,  $q>\binom{n}{2}/r(n,k)$.

Consider an arbitrary Red/Blue coloring $c'$ of $K_n$ and a weight
assignment to its edges $w'$ such that the total weight on any
mono-$k$-subgraph is at most $1$ and the total sum of the weights on
the edges of $K_n$ is $r'$. Construct a new weight function $w''$,
$w''(e)=w'(e)\binom{n}{2}/r(n,k).$ Then, with respect to $w''$, any
mono-$k$-subgraph in $c'$ has weight at most $\binom{n}{2}/r(n,k)$.
The total sum of the weights on the edges of $K_n$ is
$r'\binom{n}{2}/r(n,k)>\binom{n}{2}$. Thus, $r'> r(n,k)$, a
contradiction to the definition of $r(n,k)$.

To show a lower bound, assume that $\wram(n,k)<\binom{n}{2}/r(n,k)$.
This means that there is a Red/Blue coloring $c$ of $K_n$ and a
weight assignment $w$ to its edges such that each mono-$k$-subgraph
has sum of weights on its edges at most $q$,
$q<\binom{n}{2}/r(n,k)$. Consider a new weight assignment $w'$,
$w'(e)=w(e)/q$. Then the sum of weights $w'$ in each
mono-$k$-subgraph of $c$ is at most $1$. Moreover, the total sum of
the weights is $\binom{n}{2}/q > r(n,k)$, a contradiction to the
definition of $r(n,k)$.
\end{proof}

\begin{proposition}
   Let $k,\ell,n$ be integers, $3\leq k\leq\ell\leq n$.  Then
   $$ \wram(\ell,k)\leq\wram(n,k)\leq\binom{k}{2} . $$
   \label{prop:monotone}
\end{proposition}

\begin{proof}
Note that to prove the lower bound on $\wram(n,k)$,  it is sufficient to prove that
$$r(n,k) \leq r(\ell, k) \frac{\binom{n}{2}}{\binom{\ell}{2}}.$$

Consider a Red/Blue coloring $c$ of $K_n$. Let $w$ be a weight
function giving an optimal solution of (\ref{r}). By adding up the
sums of weights on complete $\ell$-vertex subgraphs, we have that
$$ r(c; n,k)\leq
   r(c; \ell, k)\frac{\binom{n}{\ell}}{\binom{n-2}{\ell-2}}\leq
   r(\ell, k)\frac{\binom{n}{2}}{\binom{\ell}{2}}\leq
   r(\ell,k)\frac{\binom{n}{2}}{\binom{\ell}{2}}. $$

  The upper bound is obvious by assigning weight $1$ to each edge of $K_n$.
\end{proof}

Now, since function $\wram(n,k)$ is monotone in $n$, and bounded,
the weighted Ramsey limit is well-defined.

\section{Proof of Theorem \ref{general-our}}\label{general}

\newcommand{\ex}{{\rm ex}}
We shall need the following definitions in this section. The
classical Ramsey number, $R(i)$, is the smallest number of
vertices in a complete graph such that any  Red/Blue edge-coloring
contains a monochromatic complete subgraph on $i$ vertices.  The Tur\'an
graph $T(n,i)$  is a complete $i$-partite graph on $n$
vertices with parts of almost equal sizes (differing by at most
one), its size is denoted $t(n,i)$.  For a graph $H$, let $\ex(n,H)$ be the largest number of
edges in an $n$-vertex graph which has no subgraph isomorphic to
$H$.  Tur\'an's theorem states that $\ex(n, K_{i+1})= t(n,i)$. For
a  complete graph on vertices $v_1, \ldots, v_m$, edge-colored
with a coloring $c$, we say that a colored complete $m$-partite
graph with parts $V_1, \ldots, V_m$ is a {\it balanced blow-up of
c}, if the sizes of parts $V_1, \ldots, V_m$ differ by at most $1$
and the color of all edges between $V_i$ and $V_j$ is equal to
$c(v_i, v_j)$, $1\leq i< j\leq m$.

Finally, we shall use the following Tur\'an-type implication of
the degree form of Szemer\'edi's regularity lemma, see, for
example \cite{KSSSz}.
\begin{lemma}\label{reg}
For a fixed integer  $s$, fixed $\epsilon>0$, there is an
$n_0=n_0(s,\epsilon)$, such that for all $n$, $n\geq n_0$, the
following holds: Let $G$ be an $n$-vertex graph edge-colored with
Red and Blue. If the number of edges in $G$ is greater than  $t(n,
R(i)-1)+\epsilon n^2$, then $G$ has a complete  $i$-partite
monochromatic subgraph with at least $s$ vertices in each part.
\end{lemma}

We will use Lemma~\ref{reg} in Section~\ref{sec:lb} to find a
lower bound on $\wram(n,k)$. First we show a construction which
gives an upper bound on $\wram(n,k)$.

\subsection{Upper bound on $\wram(n,k)$}
We need to find an appropriate  coloring for the edges of $K_n$
and a weight assignment function  providing a feasible solution to the
linear program (\ref{r}).

Let $k=4$.  Let the Red edges form a copy of $K_{\lfloor
n/2\rfloor,\lceil n/2\rceil}$ and let all other edges be Blue. Let
each  Red edge have weight $1/4$ and let each Blue edge have
weight $1/6$.  It is easy to see that this assignment satisfies constraints in (\ref{r}), i.e., it is a feasible solution
of that linear program.
The sum of the weights on all edges is

$$ r=  \frac{5}{24}\binom{n}{2}
   +\frac{1}{24}\left\lfloor\frac{n}{2}\right\rfloor.$$

By Lemma~\ref{lem:equivalence},
$$  \wram(n,4)  \leq \frac{\binom{n}{2}}{r}   \leq 4.8.$$

Now let $k\geq 5$ and $n\geq 5\lceil k/2\rceil$. Let an
edge-colored graph $G$ on $n$ vertices  be a balanced blow-up of a
$2$-edge colored $K_5$ with no monochromatic triangles. Let $G$ have parts $V_1, \ldots, V_5$.  Give the edges of $G$
weight $\lfloor k^2/4\rfloor^{-1}$.

Arbitrarily color the edges inside of $V_i$, for $i=1,\ldots,5$.
Give these edges weight $0$.  Since $G$ has no monochromatic
triangles, Tur\'an's theorem implies  that each mono-$k$-subgraph
has at most $\lfloor k^2/4\rfloor$ edges. Hence, this weight
assignment gives a feasible solution to (\ref{r}) with respect to
constructed coloring. The total weight is $t(n,5)\lfloor
k^2/4\rfloor^{-1}$. Therefore, $\wram(n,k)\leq
\frac{\binom{n}{2}}{t(n,5)}\left\lfloor\frac{k^2}{4}\right\rfloor\leq
1.25\left\lfloor\frac{k^2}{4}\right\rfloor$.

\subsection{Lower bound on $\wram(n,k)$}
\label{sec:lb}

Consider a weight function $w$ on the edges of $K_n$ colored in Red
and Blue with coloring $c$, such that for any mono-$k$-subgraph, the
sum of weights on its edges is at most $1$. We shall give the upper
bound on the total weight on all edges in $K_n$ by showing that one
cannot have too many ``heavy'' edges. This will give an upper bound
on $r(n,k)$ and, therefore, a lower bound on $\wram(n,k)$.  Let
$G(i)$ be a spanning graph of $K_n$ with edges of weight strictly
greater than $i$. Let $E(i)= E(G(i))$, $\overline{E(i)}=
E(K_n)\setminus E(i)$ for all $i$. Then we have, for any integers
$i_1 \leq i_2 \leq \cdots \leq i_m$ that
$$ E(K_n)=E(1/i_1)\cup\bigl(E(1/i_2)\setminus E(1/i_1)\bigr)
   \cup\cdots\cup\bigl(E(1/i_m)\setminus E(1/i_{m-1})\bigr)\cup
   \overline{E(1/i_m)} . $$
We shall consider such a partition of the edge set of $K_n$ such
that each $i_j$ corresponds to a Tur\'an number; i.e.,
$i_1=t(k,2)$, $i_2=t(k,3)$, etc.

{\it Claim 1.} $|E(1/t(k,2))|= o(n^2)$. \\
Indeed, let $G$ be a monochromatic  subgraph of $G(1/t(k,2))$.
$G$ has no subgraph isomorphic to $K_{\lfloor
k/2\rfloor,\lceil k/2\rceil}$ since otherwise this subgraph will
have weight greater than $1$.
 Therefore, using the fact that $\ex(n; K_{\lceil k/2\rceil, \lfloor
k/2 \rfloor}) \leq cn^{2-2/k}$ (see, for example, Chapter $6$
in~\cite{B}) we have the desired
result. \\

{\it Claim 2.} $|E(1/t(k,i))|\leq (1+o(1))\left( 1 -
\frac{1}{R(i)-1}\right)\binom{n}{2}$, for all $i\geq 3$. \\
Assume the opposite, then Lemma~\ref{reg} implies that
$G(1/t(k,i))$ has a monochromatic complete $i$-partite subgraph
with at least $k$ vertices in each part. Thus, $G(1/t(k,i))$ has a
monochromatic copy, $T$, of $T(k,i)$.  Since the weight of each
edge in $T$ is greater than $1/t(k,i)$, the total
weight on this subgraph is greater than $1$, a contradiction. This
proves Claim 2.

Now, we are ready to write down the expression of the total weight
on edges of $K_n$ giving an upper bound on $r(n,k)$. Since each
edge has weight at most 1,

\begin{eqnarray}
   r(n,k) & \leq & \left|E\left(1/t(k,2)\right)\right|\cdot 1
   +\sum_{i=2}^{k-1}\left(\left|E\left(1/t(k,i+1)\right)\right|
                      -\left|E\left(1/t(k,i)\right)\right|\right)
                \frac{1}{t(k,i)} \nonumber \\
   & & +\left(\binom{n}{2}-\left|E\left(1/t(k,k)\right)\right|\right)
    \frac{1}{t(k,k)} \nonumber \\
   & = & \left(1-\frac{1}{t(k,2)}\right)
         \left|E\left(1/t(k,2)\right)\right|
         +\sum_{i=3}^{k}
         \left(\frac{1}{t(k,i-1)}-\frac{1}{t(k,i)}\right)
         \left|E\left(1/t(k,i)\right)\right| \nonumber \\
   & &   +\frac{1}{t(k,k)}\binom{n}{2} \nonumber \\
   & \leq & o(n^2)+(1+o(1))\sum_{i=3}^{k}
         \left(\frac{1}{t(k,i-1)}-\frac{1}{t(k,i)}\right)
         \left(1-\frac{1}{R(i)-1}\right)\binom{n}{2}
         +\frac{1}{t(k,k)}\binom{n}{2} \nonumber \\
   & = & o(n^2)+\left(\frac{1}{t(k,2)}
         -\sum_{i=3}^k\frac{1}{R(i)-1}
          \left[\frac{1}{t(k,i-1)}-\frac{1}{t(k,i)}\right]\right)
          \binom{n}{2} \nonumber \\
   & = & (1+o(1))\frac{\binom{n}{2}}{t(k,2)}
         \left(1-\sum_{i=3}^k\frac{1}{R(i)-1}
                             \left[\frac{t(k,2)}{t(k,i-1)}
                                   -\frac{t(k,2)}{t(k,i)}\right]
         \right) . \label{turanbound}
\end{eqnarray}

Thus, for $j\in\{3,\ldots,k\}$, the following holds
\begin{equation}
   r(n,k)\leq(1+o(1))\frac{\binom{n}{2}}{t(k,2)}
         \left(1-\sum_{i=3}^j\frac{1}{R(i)-1}
                             \left[\frac{t(k,2)}{t(k,i-1)}
                                   -\frac{t(k,2)}{t(k,i)}\right]
         \right) . \label{rnkbound}
\end{equation}

Let us denote the terms used in summation as follows.
$$ \alpha(k,i)\stackrel{\rm def}{=}
   \frac{t(k,2)}{t(k,i-1)}-\frac{t(k,2)}{t(k,i)}. $$
Let $UR(i)$ be an upper bound on $R(i)-1$.  For $j=\min\{k,8\}$,
denote the  expression given in parentheses in (\ref{rnkbound}), divided
by $t(k,2)$, by $c(k)$. That is,
$$ c(k)\stackrel{\rm def}{=}\frac{1}{t(k,2)}
       \left(1-\sum_{i=3}^{\min\{k,8\}}\frac{1}{UR(i)}\alpha(k,i)\right) . $$

Then, we have that for any $j\leq k$, \begin{equation} r(n,k)\leq
(1+o(1))\binom{n}{2}c(k) . \label{ckbound} \end{equation}

We use the values of $UR(i)$, $i=3,\ldots,8$, provided by
Appendix~\ref{Ramsey} and the values of $\alpha(k,i)$ provided by
Appendix~\ref{Turan}.  For $k=3,\ldots,8$, the values of
$\alpha(k,i)$ are found by looking them up in a table. There is a
general lower bound for $\alpha(k,i)$ when $k\geq 9$, which gives
a better result for larger $k$. We summarize the upper bounds on
$c(k)$ in the following table.

\begin{center}
\begin{tabular}{|r||r|r|r|r|r|r|r|} \hline
   $k$    & $4$      & $5$
          & $6$                     & $7$
          & $8$            & $\geq 9$
          &  large enough \\ \hline
   \rule{0mm}{12pt}
   $c(k)$ & $0.2381$ & $\frac{0.9438}{t(5,2)}$
          & $\frac{0.9454}{t(6,2)}$ & $\frac{0.9442}{t(7,2)}$
          & $\frac{0.9461}{t(8,2)}$ & $\frac{0.95143}{t(k,2)}$
          & $ \frac{0.9441}{t(k,2)}$ \\ \hline
\end{tabular}
\end{center}

Using the exact values on Tur\'an numbers, in
Appendix~\ref{Turan}, and the fact that
$\wram(n,k)\geq\frac{1}{c(k)}(1+o(1))$, we conclude the proof of
the lower bound of $\wram(n,k)$ for $k\geq 4$.

\section{Equivalence of fractional packing and constraint
weigh assignment  in graphs with respect to $3$-vertex
subgraphs}\label{triangles}

\def\T{{\mathcal T}}
\def\TILT{\tilde{\mathcal T}}
\def\tilt{\tilde{t}}

Let $G$ be a graph.  We define $\tau(G)$ to be the triangle
packing number; i.e., the size of the largest edge-disjoint family
of triangles in $G$.  Its fractional relaxation is $\tau^*(G)$, as
defined in the introduction. Let $\T(G)$, $\TILT(G)$, $\T_3(G)$,
be the sets of: all induced $3$-vertex subgraphs of $G$, all
$3$-vertex subgraphs of $G$, all complete $3$-vertex subgraphs of
$G$, respectively. Observe that
$\T_3(G)\subseteq\T(G)\subseteq\TILT(G)$.

In order to establish the equivalence we want, we need to define
the following graph invariants:
\begin{eqnarray}
   r(G) & = & \min\;\;\sum\limits_{T\in \T(G)}t(T)
   \label{LP:rG} \\
   & & \begin{array}{rl}
          \mbox{such that } &
          \left\{\begin{array}{rl}
                    \sum\limits_{\scriptsize
                                 \begin{array}{l} T\ni e \\
                                 T\in\T(G) \end{array}} t(T)\geq 1, &
                 \forall e\in E(G) ; \\
                    t(T)\geq 0, & \forall T\in \T(G)
                 \end{array}\right.
       \end{array} \nonumber \\
   \tilde{r}(G) & = & \min\;\;\sum\limits_{T\in \TILT(G)}
   \tilt(T) \label{LP:tilrG} \\
   & & \begin{array}{rl}
          \mbox{such that } &
          \left\{\begin{array}{rl}
                    \sum\limits_{\scriptsize
                                 \begin{array}{l} T\ni e \\
                                 T\in\TILT(G) \end{array}}
                                 \tilt(T)=1, &
                    \forall e\in E(G) ; \\
                       \tilt(T)\geq 0, & \forall T\in\TILT(G)
                 \end{array}\right. \nonumber
      \end{array}
\end{eqnarray}

We prove the following in Appendix~\ref{sec:rrtil}.
\begin{lemma}
   Let $G$ be a graph,  then $r(G)=\tilde{r}(G)$.
   \label{lem:rrtil}
\end{lemma}


\begin{lemma}
   Let $G$ be a graph on $n\geq 3$ vertices with $e(G)$ edges and
   fractional triangle packing number $\tau^*(G)$.  Then,
   $$ \frac{1}{2}e(G)-\frac{1}{2}\tau^*(G)\leq r(G)
      \leq\frac{1}{2}e(G)-\frac{1}{2}\tau^*(G)
      +\left\lfloor\frac{n}{2}\right\rfloor . $$
   \label{lem:taustar}
\end{lemma}

\begin{proof}
   Let  $\tilt^*$ be an optimal solution of (\ref{LP:tilrG}).
   We shall construct a feasible solution of (\ref{LP:taustar}),
   giving a lower bound on $\tau^*$.  Let
   $g(T)=\tilt^*(T)$ if $T\in\T_3(G)$ and let $g(T)=0$ otherwise.
   Observe first that
   \begin{eqnarray*}
      e(G) & = & \sum_{e\in E(G)}
      \sum_{\scriptsize \begin{array}{l} T\ni e \\ T\in\TILT(G)
                        \end{array}}\tilt^*(T)
   \end{eqnarray*}

   Note that each member of $\T_3(G)$ appears in three sums
   of the form $\sum_{T\ni e, T\in\T_3(G)}\tilt^*(T)$.  In addition, each member of
   $\TILT(G)\setminus\T_3(G)$ appears in at most two of the sums
   of the form $\sum_{T\ni e,
   T\in\TILT(G)\setminus\T_3(G)}\tilt^*(T)$.  Thus,
   \begin{eqnarray*}
      e(G) & = & \sum_{e\in E(G)}
      \sum_{\scriptsize \begin{array}{l} T\ni e \\ T\in\TILT(G)
                        \end{array}}\tilt^*(T) \\
      & = & \sum_{e\in E(G)}
      \sum_{\scriptsize \begin{array}{l} T\ni e \\ T\in\T_3(G)
                        \end{array}}\tilt^*(T)
      +\sum_{e\in E(G)}
      \sum_{\scriptsize \begin{array}{l} T\ni e \\
                           T\in\TILT(G)\setminus\T_3(G)
                        \end{array}}\tilt^*(T) \\
      & = & \sum_{e\in E(G)}
      \sum_{\scriptsize \begin{array}{l} T\ni e \\ T\in\T_3(G)
                        \end{array}}g(T)
      +\sum_{e\in E(G)}
      \sum_{\scriptsize \begin{array}{l} T\ni e \\
                           T\in\TILT(G)\setminus\T_3(G)
                        \end{array}}\tilt^*(T) \\
      & \leq & 3\sum_{T\in\T_3(G)}g(T)
      +2\sum_{T\in\TILT(G)}\left(\tilt^*(T)-g(T)\right) \\
      & \leq & \sum_{T\in\T_3(G)}g(T)+2\sum_{T\in\TILT(G)}\tilt^*(T) \\
      & \leq & \tau^*(G)+2r(G)
   \end{eqnarray*}

   As a result, $r(G)\geq\frac{1}{2}e(G)-\frac{1}{2}\tau^*(G)$.

   \newcommand{\defi}{{\rm def}}
   For the other direction, let $g^*$ be an optimal solution of
   (\ref{LP:taustar}).  We shall construct $\tilt$, a feasible
   solution of (\ref{LP:tilrG}), from $g^*$ via the following
   algorithm.  For a weight function, $\nu$, defined on
   $\TILT(G)$, we define the {\it deficiency of an edge $e$
   with respect to $\nu$}, to be
   $\defi(\nu,e)=1-\sum_{T\ni e, T\in\TILT(G)}\nu(T)$. We say
   that an edge is {\it underweight} with respect to $\nu$ if
   $\defi(\nu,e)>0$.

   {\bf Initialization.} Let
   \begin{equation}
      \tilt(T)=\tilt_0(T)=
      \begin{cases}
         g^*(T), & T\in \T_3(G) ;\\
         0, & \mbox{otherwise}.
      \end{cases}
   \end{equation}

   {\bf Iteration.} Let $U$ be a set of underweight edges with
   respect to $\tilt$. Since $g^*$ is optimal, the edges in
   $U$ do not have triangles. Let
   $$ U=\left(\{e_1,e_1'\}\cup\cdots\cup\{e_m, e_m'\}\right)
      \cup\left(\{e_{m+1},\ldots, e_u\}\right) , $$
   such that $e_1, e_1', \ldots, e_m, e_m', e_{m+1}, \ldots, e_u$ are distinct edges;
    $e_i, e_i'$ are adjacent, $i=1,\ldots,m$, and $m$ is as large
   as possible. Let $T_i\in\TILT(G)$ be a subgraph with two edges
   $e_i,e_i'$, and assume also that
   $\defi(\tilt,e_i)<\defi(\tilt, e_i')$, for $i=1,\ldots,m$.
   Let
   $$\tilt'(T)=
\begin{cases}
\defi(\tilt,e_i), & \mbox{if } T=T_i ; \\
 \tilt(T), & \mbox{otherwise}.
\end{cases}$$

   Note that
   $$ \sum_{T\in\TILT(G)}\tilt'(T)
      =\sum_{T\in\TILT(G)}\tilt(T)+\sum_{i=1}^{m}\defi(\tilt,e_i)
      . $$
   Moreover, $\defi(\tilt',e_i)=0$, and $\defi(\tilt',e_i')
   =\defi(\tilt,e_i')-\defi(\tilt,e_i)$, for $i=1,\ldots,q$.
   For all other edges, the deficiencies are not changed. Let
   $\tilt(T)=\tilt'(T)$, $T\in\TILT$.

   {\bf Termination.} Stop when the set of edges that are
   underweight, with respect to $\tilt$, is a matching,
   $\{e_1,\ldots,e_q\}$.  Note that
   $q\leq\lfloor n/2\rfloor$.  Let $T_i\in\TILT(G)$
   be a graph formed by a single edge $e_i$ and a single vertex
   $v$, $i=1,\ldots,q$.  Let $\T_1 =\{T_1,\ldots,T_q\}$.
   Let $\tilt(T_i):=\defi(e_i)\leq 1$, for $i=1, \ldots, q$, let
   $\T_2(G)$ be the set of three-vertex, $2$-edge-subgraphs of
   $G$.

   We have that
   $$ \sum_{T\in\TILT(G)}\tilt(T)\leq
      \sum_{T\in\T_3(G)}g^*(T)+
      \sum_{T\in\T_2(G)}\tilt(T) +
      \sum_{T\in\T_1(G)}\tilt(T) . $$

   Note that, for a fixed edge $e$ of $G$,
   $$ \sum_{\scriptsize \begin{array}{l} T\in\T_2(G) \\
            e\in E(T)\end{array}}\tilt(T)\leq\defi(\tilt_0, e) . $$
   We also have, since each $T\in\T_2(G)$ contains exactly two
   edges, that
   \begin{eqnarray*}
      2\sum_{T\in\T_2(G)}\tilt(T) & = &
      \sum_{T\in\T_2(G)}\sum_{e\in E(T)}\tilt(T)
      =\sum_{e\in E(G)}\sum_{\scriptsize \begin{array}{l} T\ni e \\
                             T\in\T_2 \end{array}}\tilt(T) \\
      & \leq & \sum_{e\in E(G)}\defi(\tilt_0, e)
      =\sum_{e\in E(G)}
       \left[1-\sum_{\scriptsize \begin{array}{l} T\ni e \\
                     T\in\T_3(G) \end{array}}
               g^*(T)\right] .
   \end{eqnarray*}

   Therefore,
   \begin{eqnarray*}
      \sum_{T\in\TILT(G)}\tilt(T) & \leq &
      \sum_{T\in\T_3(G)}\tilt(T)+\sum_{T\in\T_2(G)}\tilt(T)
      +\sum_{T\in\T_1(G)}\tilt(T) \\
      & \leq & \sum_{T\in\T_3(G)}g^*(T)+
      \frac{1}{2}\sum_{e\in E(G)}
      \left[1-\sum_{\scriptsize \begin{array}{l} T\ni e \\
                    T\in\T_3(G) \end{array}}g^*(T)\right]
      +\sum_{T\in\T_1(G)}1 \\
      & \leq & \frac{1}{2}e(G)
      -\frac{1}{2}\sum_{T\in\T_3(G)}g^*(T)
      +\left\lfloor\frac{n}{2}\right\rfloor \\
      & = & \frac{1}{2}e(G)-\frac{1}{2}\tau^*(G)
      +\left\lfloor\frac{n}{2}\right\rfloor .
   \end{eqnarray*}

 Recall from (\ref{LP:rG}) that  $r$ computes a minimum. Therefore
   $$ r(G)\leq\frac{1}{2}e(G)
      -\frac{1}{2}\tau^*(G)
      +\left\lfloor\frac{n}{2}\right\rfloor . $$
\end{proof}

\appendix
\section{Proof of Lemma~\ref{lem:rrtil}}
\label{sec:rrtil}

\subsection{$r(G)\leq\tilde{r}(G)$}

Let $\tilt$ be a feasible solution of (\ref{LP:tilrG}).  For each
$T\in\T(G)$, let $t(T)=\sum_{S\subseteq T, S\in\TILT(G)}\tilt(S)$.
This ensures that, for all $e\in E(G)$,
$$ \sum_{T\ni e,T\in\T(G)}t(T)\geq 1 . $$
Since any $S\in\TILT(G)$ is in a unique $T\in\T(G)$, we have
$$ \sum_{T\in\T(G)}t(T)=\sum_{T\in\TILT(G)}\tilt(T) . $$

Since both linear programs compute a minimum, if $\tilt^*$ is an
optimal solution to (\ref{LP:tilrG}) and $t^*$ is the
corresponding solution to (\ref{LP:rG}) as computed above, then
$$ r(G)\leq\sum_{T\in\T(G)}t^*(T)=\sum_{T\in\TILT(G)}\tilt^*(T)
   =\tilde{r}(G) . $$

\subsection{$\tilde{r}(G)\leq r(G)$}

\newcommand{\excess}{{\rm exc}}
Let $t$ be a minimal feasible solution of (\ref{LP:rG}).  We shall
create a feasible solution, $\tilt$, of (\ref{LP:tilrG}) by
redistributing the weights on $\T(G)$ to $\TILT(G)$, such that the
total weights on edges become equal to one. For any weight
function $\nu$ on $\TILT(G)$, define the {\em excess of an edge
$e$ with respect to $\nu$} to be $\excess(\nu,e)=\left(\sum_{S\ni
e, S\in\TILT(G)}\nu(S)\right)-1$. We call an edge $e$ {\em
overweight} with respect to $\nu$ if $\excess(\nu,e)>0$. We define
$\tilt$ via the following algorithm.

\textbf{Initialization.} Let
$$ \tilt(T)=\begin{cases}
               t(T), & T\in\T(G) ; \\
               0, & \mbox{otherwise.}
            \end{cases} $$
Observe that the total weight is as follows:
$$ \sum_{T\in\TILT(G)}\tilt(T)=\sum_{T\in\T(G)}\tilt(T)
   =\sum_{T\in\T(G)}t(T).$$

\textbf{Iteration.} Consider some $T\in\TILT(G)$ containing an
overweight edge with respect to $\tilt$ such that $\tilt(T)>0$.
Note that as long as edges with positive excess exist, such a $T$
exists as well.

If $T$ contains $e$ as its only overweight edge, then let
\begin{eqnarray*}
   \tilt'(T) & = & \tilt(T)-\min\left\{\tilt(T),
   {\rm ex}(\tilt,e)\right\} , \\
   \tilt'(T\setminus e) & = & \tilt(T\setminus e)
   +\min\left\{\tilt(T),\excess(\tilt,e)\right\} .
\end{eqnarray*}
For all other $S\in\TILT(G)$, let $\tilt'(S)=\tilt(S)$.

Let $T$ contain two overweight edges, $e$ and $e'$, such that
$\excess(\tilt,e')\leq\excess(\tilt,e)$.  If
$\tilt(T)>\excess(\tilt,e)$, then let
\begin{eqnarray*}
   \tilt'(T) & = & \tilt(T)-\excess(\tilt,e) \\
   \tilt'(T\setminus e) & = & \tilt(T\setminus e)
   +\excess(\tilt,e)-\excess(\tilt,e') \\
   \tilt'(T\setminus (e\cup e')) & = &
   \tilt(T\setminus (e\cup e'))+\excess(\tilt,e') .
\end{eqnarray*}
Otherwise (i.e, if  $\tilt(T)\leq\excess(\tilt,e)$), let
\begin{eqnarray*}
   \tilt'(T) & = & 0 \\
   \tilt'(T\setminus e) & = & \tilt(T\setminus e) \\
   \tilt'(T\setminus (e\cup f)) & = &
   \tilt(T\setminus (e\cup f))+\tilt(T)
\end{eqnarray*}
For all other $S\in\TILT(G)$, let $\tilt'(S)=\tilt(S)$.

Finally, $T$ cannot have three overweight edges because $t$ was
minimal and $\excess(\tilt,e)\leq\excess(t,e)$ for all $e\in G$.

Clearly, the total weight does not change:
$$ \sum_{T\in\TILT(G)}\tilt'(T)=\sum_{T\in\TILT(G)}\tilt(T) . $$
Moreover,  $0\leq\excess(\tilt',f)\leq\excess(\tilt,f)$ for all
$f\in E(G)$ and $\sum_{e\in E(G)}\excess(\tilt',e)<\sum_{e\in
E(G)}\excess(\tilt,e)$.

Set $\tilt(T):=\tilt'(T)$ for all $T\in\TILT(G)$.

\textbf{Termination.} Stop if $\excess(\tilt,e)=0$ for all $e\in
E(G)$.

To see that the process terminates, observe that at each iteration of this
procedure, we either reduce the number of overweight edges or we
both (1) reduce the sum $\sum_{e\in G}{\rm ex}(\tilt,e)$ by at
least $m(\tilt):=\min\{\tilt(T):\tilt(T)>0\}$ and (2) ensure that
each $\tilt'(T)$ will either remain the same, be zero or increase
by at least $m(\tilt)$.  So, $m(\tilt')\geq m(\tilt)$. Therefore,
each iteration of the algorithm will decrease $\sum_{e\in G}{\rm
ex}(\tilt,e)$ by a fixed amount until the number of overweight
edges decreases. \\

\textbf{Concluding the proof.} At the end of this procedure, we
have a feasible solution $\tilt_0$ of (\ref{LP:tilrG}) such that
$\sum_{T\in\TILT(G)}\tilt_0(T)=\sum_{T\in\T(G)}t(T)$.  Since both
linear programs compute a minimum, if $t^*$ is an optimal solution
to (\ref{LP:rG}) and $\tilt_0^*$ is the corresponding solution to
(\ref{LP:tilrG}) as computed above, then
$$ \tilde{r}(G)\leq\sum_{T\in\TILT(G)}\tilt_0^*(T)
   =\sum_{T\in\T(G)}t^*(T)=r(G) . $$

\section{Bounds on Ramsey numbers}
\label{Ramsey}

\begin{figure}[ht]
\begin{center}
\begin{tabular}{|r||c|c|c|c|c|c|} \hline
   $i$    & 3 & 4  & 5 & 6 & 7 & 8  \\ \hline
   $UR(i)$ & 5 & 17 & 48 & 164 & 539 & 1869
   \\ \hline
\end{tabular}
\end{center}
\caption{Known values for $UR(i)$, an upper bound on $R(i)-1$,
from~\cite{R}.} \label{fig:Ramsey}
\end{figure}

\section{Tur\'an numbers}
\label{Turan}

The Tur\'an number $t(k,i)$, for $k\geq 3$ and $i=2,\ldots,k$, can
be computed exactly to be
$$ t(k,i)=\frac{k^2}{2}\left(\frac{i-1}{i}\right)
          -\frac{i}{2}\left(\left\lceil\frac{k}{i}\right\rceil
                            -\frac{k}{i}\right)
                      \left(\frac{k}{i}
                            -\left\lfloor\frac{k}{i}
                             \right\rfloor\right) . $$

Clearly, $t(k,i)\leq\frac{k^2}{2}\left(\frac{i-1}{i}\right)$ and
$$ t(k,i)\geq\frac{k^2}{2}\left(\frac{i-1}{i}\right)
   -\frac{1}{2i}\left\lfloor\frac{i^2}{4}\right\rfloor
   \geq\frac{k^2}{2}\left(\frac{i-1}{i}\right)
   -\frac{i}{8} . $$

Figure~\ref{fig:Turan} gives the exact values for small Tur\'an
numbers.
\begin{figure}[ht]
\begin{center}
\begin{tabular}{rr||r|r|r|r|r|r|r|}
   \multicolumn{2}{r||}{}         & \multicolumn{6}{c|}{$k$} \\
   \multicolumn{2}{r||}{$t(k,i)$} & 3 & 4 & 5  & 6  & 7  & 8 \\ \hline
       & 2    & 2 & 4 & 6  & 9  & 12 & 16 \\ \cline{2-8}
       & 3    & 3 & 5 & 8  & 12 & 16 & 21 \\ \cline{2-8}
       & 4    &   & 6 & 9  & 13 & 18 & 24 \\ \cline{2-8}
   $i$ & 5    &   &   & 10 & 14 & 19 & 25 \\ \cline{2-8}
       & 6    &   &   &    & 15 & 20 & 26 \\ \cline{2-8}
       & 7    &   &   &    &    & 21 & 27 \\ \cline{2-8}
       & 8    &   &   &    &    &    & 28 \\ \hline
\end{tabular}
\end{center}
\caption{Tur\'an numbers, $t(k,i)$, $k\leq 8$.} \label{fig:Turan}
\end{figure}

The number $\alpha(k,i)$ is used in Section~\ref{sec:lb}.  Recall
that for $3\leq i\leq k$,
$$ \alpha(k,i)=\frac{t(k,2)}{t(k,i-1)}-\frac{t(k,2)}{t(k,i)} . $$
Figure~\ref{fig:alpha} gives exact values for $\alpha(k,i)$ for
small values of $k$.
\begin{figure}[ht]
\begin{center}
\begin{tabular}{rr||r|r|r|r|r|r|r|}
   \multicolumn{2}{r||}{}         & \multicolumn{6}{c|}{$k$} \\
   \multicolumn{2}{r||}{$\alpha(k,i)$} & 3 & 4
   & 5              & 6               & 7
   & 8 \\ \hline
       & 3 & $1/3$ & $1/5$  & $1/4$  & $1/4$   & $1/4$
       & $5/21$ \\ \cline{2-8}
       & 4 &       & $2/15$ & $1/12$ & $3/52$  & $1/12$
       & $2/21$ \\ \cline{2-8}
   $i$ & 5 &       &        & $1/15$ & $9/182$ & $2/57$
       & $2/75$ \\ \cline{2-8}
       & 6 &       &        &        & $3/70$  & $3/95$
       & $8/325$ \\ \cline{2-8}
       & 7 &       &        &        &         & $1/35$
       & $8/351$ \\ \cline{2-8}
       & 8 &       &        &        &         &
       & $4/189$ \\ \hline
\end{tabular}
\end{center}
\caption{Values of $\alpha(k,i)$, $k\leq 8$.} \label{fig:alpha}
\end{figure}

For $k\geq 9$, we determine a lower bound on  $\alpha(k,i)$.
\begin{eqnarray}
   \alpha(k,i) & \geq &
   \frac{\left\lfloor\frac{k^2}{4}\right\rfloor}
        {\frac{k^2}{2}\left(\frac{i-2}{i-1}\right)}
   -\frac{\left\lfloor\frac{k^2}{4}\right\rfloor}
         {\frac{k^2}{2}\left(\frac{i-1}{i}\right)
          -\frac{i}{8}} \nonumber \\
   & = & \frac{2}{k^2}\left\lfloor\frac{k^2}{4}\right\rfloor
         \left(\frac{i-1}{i-2}
         -\frac{i}{i-1}\frac{1}{1-\frac{i^2}{4k^2(i-1)}}\right)
         \nonumber \\
   & \geq & \frac{2}{k^2}\left\lfloor\frac{k^2}{4}\right\rfloor
            \left(\frac{i-1}{i-2}
            -\frac{i}{i-1}
             \left(1+\frac{1}{4k-5}\right)\right) \nonumber \\
   & \geq & \frac{2}{k^2}\left\lfloor\frac{k^2}{4}\right\rfloor
            \left(\frac{1}{(i-1)(i-2)}
            -\frac{i}{i-1}\frac{1}{4k-5}\right). \label{alpha}
\end{eqnarray}

Substituting (\ref{alpha}) into (\ref{ckbound}), we obtain the
following for $k\geq 9$:

\begin{eqnarray}
r(n,k)&\leq &(1+o(1))\binom{n}{2}c(k) \nonumber \\
& \leq & (1+o(1))\frac{\binom{n}{2}}{t(k,2)}
                   \left(1-\sum_{i=3}^8\frac{1}{UR(i)}\alpha(n,k)
                   \right) \nonumber \\
& \leq & (1+o(1))\frac{\binom{n}{2}}{t(k,2)}
                  \left( 1 - \sum_{i=3}^8\frac{1}{UR(i)}
                             \frac{2}{k^2}
                             \left\lfloor\frac{k^2}{4}\right\rfloor
                             \frac{1}{(i-1)(i-2)}\right. \nonumber \\
& & \hspace{1.3in}\left. +\sum_{i=3}^8\frac{1}{UR(i)}
\frac{2}{k^2}\left\lfloor\frac{k^2}{4}\right\rfloor  \frac{i}{i-1}
\frac{1}{4k-5}
                  \right) \nonumber \\
& \leq &  (1+o(1))\frac{\binom{n}{2}}{t(k,2)}\left(
              1-\frac{2}{k^2}\left\lfloor\frac{k^2}{4}\right\rfloor 0.11191
              +\frac{2}{k^2} \left\lfloor\frac{k^2}{4}\right\rfloor \frac{0.41457}{4k-5}
              \right) \nonumber \\
& \leq & (1+o(1))\frac{\binom{n}{2}}{t(k,2)}
\left(.94405+\frac{0.05596}{k^2} +\frac{0.20729}{4k-5}\right) \label{rnk}
\end{eqnarray}

The  expression in (\ref{rnk}) given in parentheses  is bounded above by
$0.9515$ for all $k\geq 9$ and bounded above by $0.9441$ for $k$
large enough.
\end{document}